\documentclass[reqno,11pt]{amsart}
\usepackage{amsmath,amsthm,amsfonts,amssymb}
\usepackage[top=1in, bottom=1.25in, left=1.25in, right=1.25in]{geometry}

\usepackage{hyperref}
\usepackage[all]{xy}
\usepackage{graphicx}
\usepackage{amscd}
\usepackage{color}
\usepackage{enumerate,enumitem}
\usepackage{fourier}
\usepackage{tikz}
\usepackage{cases}

\hypersetup{
    colorlinks=true,
    linkcolor=blue,
    citecolor=green,
    urlcolor=cyan
}

\numberwithin{equation}{section}

\theoremstyle{plain} %default
\newtheorem{theorem}{Theorem}[section]

\newtheorem{prop}[theorem]{Proposition}

\theoremstyle{definition}
\newtheorem{definition}[theorem]{Definition}

\newtheorem{thmx}{Theorem}
\newtheorem{step}{Step}
\theoremstyle{remark}

\newcommand{\PP}{\mathbb{P}}

\newcommand{\Z}{\mathbb{Z}}

\newcommand{\C}{\mathbb{C}}

\newcommand{\HG}{{\widehat{G}}}
\newcommand{\id}{\mathrm{id}}

\newcommand{\barW}{{\overline{W}}}
\newcommand{\tildeW}{{\widetilde{W}}}
\newcommand{\Dup}{{D^{\uparrow}}}
\newcommand{\Ddown}{{D^{\downarrow}}}

\newcommand{\Fix}{{\textrm{Fix}}}

\allowdisplaybreaks

\begin{document}
\title{Twisted Koszul algebras of nonisolated hypersurface singularities}

\author{Sangwook Lee}
\address{Sangwook Lee: Department of Mathematics and Integrative Institute of Basic Science \\Soongsil University \\
369 Sangdo-ro, Dongjak-gu, Seoul, Korea}
\email{sangwook@ssu.ac.kr}
\begin{abstract}
Given a hypersurface singularity (not necessarily isolated) with a finite abelian group action, we develop a method to define an explicit product structure on the twisted Koszul algebra (whose invariant subalgebra is the orbifold Koszul algebra).
\end{abstract}
\maketitle

\section{Introduction}
Hypersurface singularities seem elementary but are still quite interesting objects due to abundant structures behind them. They have been vigorously studied in various areas such as topology, algebraic geometry, symplectic geometry and representation theory etc. 

We can generalize usual hypersurface singularity slightly, by considering a finite group action by $G$ on the affine space together with a polynomial $W$, which is invariant under the action. Such a pair $(W,G)$ is called an {\em orbifold Landau-Ginzburg model}. It is interesting in its own right, and it also arises naturally in mirror symmetry in the following manner. Suppose that a symplectic manifold $X$ has a LG mirror $W: \Bbbk^n \to \Bbbk$. If $\tilde{X}$ is a normal covering space of $X$ such that its deck transformation group $G$ is finite abelian, then it is expected that a mirror of $\tilde{X}$ is given by the orbifold LG model $(W,\HG)$. See \cite{CLe,HJL} for example.

We will focus on the computation of Hochschild type invariants among several approaches toward understanding singularities of orbifold LG models. In this aspect, isolated singularities have been investigated more than nonisolated ones so far. When the singularity is isolated, we come up with so-called {\em orbifold Jacobian algebra} as desired Hochschild invariant. There are works on orbifold Jacobian algebras as \cite{CalTu,BTW,HLL,ShkLGorb}. On the other hand, when the orbifold singularity is nonisolated its Hochschild invariant (namely orbifold Koszul algebra) is less understood. Nevertheless, nonisolated orbifold singularities also naturally arise when we study mirror symmetry of noncompact manifolds, e.g. \cite{HJL}. In such a case, the orbifold Koszul algebra is expected to be isomorphic to the symplectic cohomology of given noncompact manifold. Since symplectic cohomology is quite difficult to compute in general, it is desirable to know the structure of orbifold Koszul algebras as explicit as possible.

In this work, we investigate orbifold Koszul algebras in terms of matrix factorizations. The original reason for the author to use matrix factorizations (other than curved algebras as in previous works) is as follows: in earlier works \cite{CHL,CHL-toric} Cho-Hong-Lau discovered symplectic manifolds with prescribed Lagrangian submanifolds so that we can find their mirrors and prove (homological) mirror symmetry in geometric way. In this context we are naturally led to the category of matrix factorizations of the mirror orbifold LG model, and Hochschild invariants of the category of MFs are encoded in terms of Floer theory. (The relation between Floer theory and category of MFs was crucially used also in \cite{HJL}.)

Now we briefly recall the setting and main problem.
\begin{definition}\label{def:orbLG}
    Let $W\in \Bbbk[x_1,\cdots,x_n]$, and a finite group $G$ act on $\Bbbk[x_1,\cdots,x_n]$ such that $W$ is invariant under $G$-action. Then $(W,G)$ is called an {\em orbifold Landau-Ginzburg model}. 

    If an orbifold LG model $(W,G)$ further satisfies that 
    \begin{itemize}
        \item $G$ is abelian,
        \item for $h\in G$, $h\cdot x_i=h_ix_i$ for some $h_i\in \Bbbk^*$,
    \end{itemize}
    then we call $(W,G)$ a {\em diagonal orbifold LG model}.
\end{definition}
In this paper we will only deal with diagonal orbifold LG models. Now we want to define an analogue of Chen-Ruan cohomology for a diagonal orbifold LG model.
\begin{equation}\label{eq:sumofKoszul}
K^\bullet(W,G):= \bigoplus_{g\in G} K^\bullet(\partial W^g)\cdot \xi_g.
\end{equation}
Here, $K^*(\partial W^g)$ is the Koszul complex of the Jacobian ideal of the function $W|_{\Fix(g)}$, and $\xi_g$ is a formal generator which will be precisely defined later. It is rather easy to compute the module structure of $K^\bullet(W,G)$, but it is never clear what is the product on it. That is why we are led to consider Hochschild cohomology which has a natural associative algebra structure. We can consider a bimodule over a (curved) algebra or a category of matrix factorizations, whose Hochschild cohomology is isomorphic to $K^\bullet(W,G)$ as a module. As we mentioned above, we will use matrix factorizations. The following is the starting point:
\begin{thmx}
Let $(W,G)$ be a diagonal orbifold LG model. Let $\Delta_1$ and $\Delta_h$ be Koszul matrix factorizations of the diagonal ideal and the $h$-twisted diagonal ideal respectively (for $h\in G$), of a polynomial 
\begin{equation}\label{eq:boxminus}W\boxminus W:=W(y_1,\cdots,y_n)-W(x_1,\cdots,x_n)\in \Bbbk[x_1,\cdots,x_n,y_1,\cdots,y_n].
\end{equation}
Then we have a quasi-isomorphism of chain complexes
    \[(-)_{kos}:\bigoplus_{h\in G}hom_{MF(W\boxminus W)}(\Delta_1,\Delta_h)\stackrel{\sim}{\longrightarrow} K^\bullet(W,G).\]
\end{thmx}
It is known that $\bigoplus_{h\in G}hom_{MF(W\boxminus W)}(\Delta_1,\Delta_h)$ is isomorphic to an endomorphism algebra of a $(1\times G)$-equivariant matrix factorization of $W\boxminus W$, so we can define an intrinsic product structure on $K^\bullet(W,G)$ via the quasi-isomorphism. The remaining problem is therefore to construct a quasi-inverse of $(-)_{kos}$, and it is the main advance of this work.
\begin{thmx}
Let $W\in R=\Bbbk[x_1,\cdots,x_n]$, $S=\Bbbk[x_1,\cdots,x_n,y_1,\cdots,y_n]$ and $(W,G)$ be a diagonal orbifold LG model. For $i,j\in \{1,\cdots,n\}$ and $h\in G$, there exist polynomials $g_{ji}^h\in S$ and $f_{ji}^h \in R$ such that for
\begin{align*}
 \eta_h: S[\theta_1,\cdots,\theta_n,\partial_1,\cdots,\partial_n] & \to S[\theta_1,\cdots,\theta_n,\partial_1,\cdots,\partial_n],\nonumber \\
\theta_I \partial_J& \mapsto  \sum (-1)^{|I|} g^h_{ji} \frac{\partial\theta_I}{\partial\theta_i}\partial_j\partial_J+\sum f^h_{ji}\frac{\partial^2 \theta_I}{\partial\theta_j\partial\theta_i}\partial_J,
\end{align*}
the map $\exp(\eta_h)$ is a quasi-inverse of $(-)_{kos}$.
\end{thmx}

Finally, we will illustrate an example which is relevant in mirror symmetry. This will not only be helpful in understanding the recipe, but also reflect an interesting geometric feature of our result, in the sense that we can observe a similarity between an orbifold Koszul algebra and a symplectic cohomology ring by looking at the explicit algebra structure.

\begin{center}
    {\bf Acknowledgements}
\end{center}

This work grew out of the idea from author's joint work with Hansol Hong and Hyeongjun Jin and he would like to thank them for useful discussions and brilliant insights during the collaboration. This work was supported by the National Research Foundation of Korea (NRF) grant funded by the Korea government(MSIT)(No. 202117221032) and Basic Science Research Program through NRF funded by the Ministry of Education (2021R1A6A1A10044154).

\section{Preliminaries}

%\subsection{Categories of matrix factorizations}

Let $\theta_i$ and $\partial_i$ (for $i=1,\cdots,n$) be formal variables with $|\theta_i|=-1$, $|\partial_i|=1$ and
\begin{equation}\label{eq:cliffrel}
 \theta_i \theta_j=-\theta_j \theta_i,\;\; \partial_{i}\partial_{j}=-\partial_{j}\partial_{i},\;\; \partial_{i}\theta_j=-\theta_j\partial_{i}+\delta_{ij}
 \end{equation}
Let $I=\{i_1,\cdots,i_k\}$ be an ordered subset of $\{1,\cdots,n\}$. We will frequently use following multi-index notations:
\[ \theta_I:=\theta_{i_1}\cdots\theta_{i_k}, \quad \partial_I:=\partial_{i_1}\cdots\partial_{i_k}.\]
We recall an elementary definition in terms of above variables.
\begin{definition}
Let $A$ be a commutative ring. For a sequence $(r_1,\cdots,r_n)$ of elements in $A$, 
\[ K^\bullet(r_1,\cdots,r_n):=(A\langle \theta_1,\cdots,\theta_n \rangle, \sum_i r_i \partial_i)\]
is called the {\em Koszul complex} of $(r_1,\cdots,r_n)$.
\end{definition}
The cochain complex $K^\bullet(r_1,\cdots,r_n)$ is explicitly spelled out as 
\[\displaystyle \xymatrix{
 0 \ar[r] & A\cdot \theta_1\cdots\theta_n \ar[r]^-{\sum_i r_i\partial_i} 
 & \bigoplus\limits_{i_1<\cdots<i_{n-1}} A \cdot \theta_{i_1}\cdots\theta_{i_{n-1}} \ar[r] & \cdots \ar[r] & \bigoplus\limits_i A\cdot \theta_i \ar[r]^-{\sum_i r_i\partial_i} & A \ar[r] & 0} \]
with $\bigoplus\limits_{i_1<\cdots<i_m}A\cdot \theta_{i_1}\cdots\theta_{i_m}= K^{-m}(r_1,\cdots,r_n)$.
Koszul complexes can be used to define invariants of singularities as follows.
\begin{definition}
    Let $(W,G)$ be a diagonal orbifold LG model. For $h\in G$ define
    \[ x_i^h:=\begin{cases}
    x_i & {\textrm{ if }} h\cdot x_i=x_i\\
    0 & {\textrm{ if }} h\cdot x_i \neq x_i
    \end{cases},
    \quad W^h:=W(x_1^h,\cdots,x_n^h)\in \Bbbk[x_1^h,\cdots,x_n^h].
    \]
    Let \[ I_h:= \{i\mid h\cdot x_i \neq x_i\},\quad  I^h:=I_h^c\]
    be ordered subsets of $\{1,\cdots,n\}$ and define a sequence $ \partial W^h:=\big(\frac{\partial W^h}{\partial x_i}:i\in I^h \big)$ in $\Bbbk[x_1^h,\cdots,x_n^h]$. The {\em twisted Koszul complex} of $(W,G)$ is defined as
    \[ K^\bullet(W,G):= \bigoplus_{h\in G} K^\bullet(\partial W^h)\cdot \theta_{I_h}.\]
    Defining $G$-action on variables $\{\theta_1,\cdots,\theta_n,\partial_1,\cdots,\partial_n\}$ by
 \[ h\cdot\theta_i:=h_i^{-1} \theta_i, \quad h\cdot\partial_{i}:=h_i \partial_{i}\]
when $h\cdot x_i=h_i x_i$, we call 
\[ K^\bullet(W,G)^G=\bigoplus_{h\in G}\big( K^\bullet(\partial W^h)\cdot \theta_{I_h} \big)^G\]
the {\em orbifold Koszul complex} of $(W,G)$, where $\big( K^\bullet(\partial W^h)\cdot \theta_{I_h} \big)^G$ is the subcomplex of $G$-invariants of $K^\bullet(\partial W^h)\cdot \theta_{I_h}$.
\end{definition}
Observe that as each $h$-summand we did not just take the Koszul complex of Jacobian ideal of $W^h$, but shifted it by $\theta_{I_h}$. Such a modification will appear meaningful later.

We recall another important invariant of $W$.
\begin{definition}
A {\em matrix factorization} of $W \in R$ is a $\Z/2$-graded projective $R$-module $P=P_0\oplus P_1$ together with a morphism $d=(d_0,d_1)$ of degree 1 such that
\[ d^2=W\cdot \id.\] 
Let $\phi: (P,d) \to (Q,d')$ be a morphism of degree $j\in \Z/2$.  Define
\[ D\phi:=d'\circ \phi -(-1)^{|\phi|}\phi \circ d\]
and define the composition of  morphisms in the usual sense. This defines a dg-category of matrix factorizations $(MF(W),D,\circ)$.
\end{definition}

% \begin{definition}
% Let $W\in R$ and let $G$ be a group which acts on $R$ leaving $W$ invariant. A {\em $G$-equivariant matrix factorization} of $W$ is a matrix factorization $(P,d)$ where $P$ is equipped with an $G$-action, and $d$ is $G$-equivariant. A {\em $G$-equivariant morphism} $\phi$ between two $G$-equivariant matrix factorizations $(P,d)$ and $(Q,d')$ is an $G$-equivariant morphism of $\Z/2$-graded $R$-modules $P$ and $Q$. Again, $G$-equivariant matrix factorizations form a dg-category $(MF_G(W),D,\circ)$.
% \end{definition}

% The following object is used crucially in the computation of Hochschild invariants of matrix factorizations, because it is a kernel for the identity functor of $MF(W)$.

\begin{definition}
Let $(W,G)$ be a diagonal orbifold LG model and $S=\Bbbk[x_1,\cdots,x_n,y_1,\cdots,y_n]$. Let \[ \nabla_j W:= \frac{W(y_1,\cdots,y_j,x_{j+1},\cdots,x_n)-W(y_1,\cdots,y_{j-1},x_j,\cdots,x_n)}{y_j-x_j},\]
and let
\[ \Delta_h:=\Big( S\langle\theta_1,\cdots,\theta_n\rangle,\sum_{i=1}^n \big((y_i-hx_i)\theta_i+\nabla_i W(h\cdot x_1,\cdots,h\cdot x_n,y_1,\cdots,y_n) \cdot\partial_{i}\big)\Big)\]
be a matrix factorization of $W\boxminus W$ (in \eqref{eq:boxminus}). Here $S\langle \theta_1,\cdots,\theta_n\rangle$ is $\Z/2$-graded, by
division into even/odd degrees.
\end{definition}

Now we introduce the first key step towards our main result.

\begin{prop}\label{prop:mfkos}{\em (\cite[Proposition 2.6]{HJL})}
We have a quasi-isomorphism of chain complexes
\[ \bigoplus_{h\in G} hom_{MF(W\boxminus W)}(\Delta_1,\Delta_h) \simeq K^\bullet(W,G).\]
\end{prop}

We will need the following part of the proof of Proposition \ref{prop:mfkos}. Since $\Delta_h$ is constructed by the resolution of a shifted MCM module $S/(y_1-hx_1,\cdots,y_n-hx_n)[-n]$ over the hypersurface ring $S/(W\boxminus W)$, we have quasi-isomorphisms
\begin{align}\label{homcomplex}
\begin{split}
hom_{MF(W\boxminus W)}(\Delta_1,\Delta_h) & \simeq hom_S\big(\Delta_1,S/(y_1-hx_1,\cdots,y_n-hx_n)[-n] \big) \\
&\simeq hom_S\big(\Delta_1[n],S/(y_1-hx_1,\cdots,y_n-hx_n) \big) \\
&\simeq \Delta_1^\vee[-n] \otimes_S \big(S/(y_1-hx_1,\cdots,y_n-hx_n)\big). 
\end{split}
\end{align}
The latter chain complex in \eqref{homcomplex} is a double complex, and the proof involves typical spectral sequence argument. 

We can spell out a quasi-isomorphism explicitly as follows.
Let $\phi \in hom_{MF(W(y)-W(x))}(\Delta_1,\Delta_h)$ be a closed morphism. By \eqref{homcomplex} its cohomology class $[\phi]$ is completely determined by $pr\circ \phi$, where $pr: \Delta_h \to S\cdot \theta_1\cdots\theta_n$ is the projection map. We have a unique expression
\begin{equation}\label{eq:prphi}
pr\circ \phi=\sum_{J,K}b_{JK}\theta_J\theta_K\partial_{K^{op}} 
\end{equation}
such that $J$ and $K$ are ordered subsets of $\{1,\cdots,n\}$, $b_{JK}\in S$, and $\partial_{K^{op}}:=\partial_{k_m}\cdots\partial_{k_1}$ whereas $\theta_K=\theta_{k_1}\cdots\theta_{k_m}$. We point out that $K$ may be empty. We also define
\[ (pr\circ \phi)_+:= \sum_{J: I_h \subset J}\sum_K b_{JK}\theta_J\theta_K\partial_{K^{op}}.\]
For an $S$-module homomorphism 
$\phi=\sum f_{IJ}(y_1,\cdots,y_n,x_1,\cdots,x_n) \theta_I \partial_J$ we define
\begin{equation}\label{eq:restrFix}
 \phi|_{{\textrm{Fix}}(h)}:=\sum f_{IJ}(x_1^h,\cdots,x_n^h,x_1^h,\cdots,x_n^h)\theta_I\partial_J.
 \end{equation}
With respect to the expression \eqref{eq:prphi} and notation \eqref{eq:restrFix}, the quasi-isomorphism 
\[ (-)_{kos}: hom(\Delta_1,\Delta_h) \to K^\bullet(\partial_{x_i} W^h: i\in I^h)\cdot \theta_{I_h}\]
can be realized as
\[\phi_{kos}= \big(\sum_{J: I_h \subset J}\sum_K b_{JK}\theta_J\big)|_{\Fix(h)}. \]
Observe that $\phi_{kos}$ is determined only by $(pr\circ \phi)_+$. 

\section{Classification of closed morphisms of matrix factorizations}\label{sec:classification}

This section is devoted to the construction of a quasi-inverse of $(-)_{kos}$.
First, for $0 \leq j<i \leq n$ and $h\in G$, define
\[ \barW^h_{j,i}:=W(y_1,\cdots,y_{j},x_{j+1},\cdots,x_i,hx_{i+1},\cdots,hx_n)\in S,\]
\[ \tildeW^h_{j,i}:=W(x_1^h,\cdots,x_{j}^h,x_{j+1},\cdots,x_{i},hx_{i+1},\cdots,hx_n)\in R.\]
% where 
% \[ x_i^h=
% \begin{cases}
% x_i&{\rm if\;\;} hx_i=x_i \\
% 0& {\rm otherwise.}
% \end{cases}
% \]
(Recall that $x_i^h=x_i$ if $hx_i=x_i$ and $x_i^h=0$ otherwise.) We also define
\begin{equation}\label{eq:Wbar} \barW^h_{i,i}:=W(y_1,\cdots,y_{i},hx_{i+1},\cdots,hx_n),\end{equation}
\begin{equation}\label{eq:Wtilde} \tildeW^h_{i,i}:=W(x_1^h,\cdots,x_i^h,hx_{i+1},\cdots,hx_n).\end{equation}
For $i,j\in \{1,\cdots,n\}$ with $j<i$, define
\begin{align}\label{eq:gjifji}
\begin{split}
g^h_{ji}&:=\begin{cases}
\frac{(\barW^h_{j,i}-\barW^h_{j-1,i})-(\barW^h_{j,i-1}-\barW^h_{j-1,i-1})}{(y_j-x_j) (x_i-hx_i)} &{\rm if\;} i\in I_h, \\
\frac{\partial_i \barW^h_{j,i}-\partial_i \barW^h_{j-1,i}}{y_j-x_j} & {\rm if \;} i\in I^h, 
\end{cases} \\
 f^h_{ji}&:=
 \begin{cases}\frac{(\tildeW^h_{j,i}-\tildeW^h_{j-1,i})-(\tildeW^h_{j,i-1}-\tildeW^h_{j-1,i-1})}{(x_j-hx_j)(x_i-hx_i)}& {\rm if\;} i,j\in I_h,\\
 \frac{\partial_i \tildeW^h_{j,i}-\partial_i \tildeW^h_{j-1,i}}{x_j-hx_j} & {\rm if\;} j\in I_h{\rm \; and\;} i\in I^h,\\
 \frac{(\partial_j \tildeW^h_{j,i}-\partial_i \tildeW^h_{j,i-1})- ({\partial_j \tildeW^h_{i-1,i-1}-\partial_j \tildeW^h_{i,i}})}{x_i-hx_i}& {\rm if\;} j\in I^h {\rm \; and\;} \; i\in I_h,\\
 0 & {\rm otherwise}
 \end{cases}
 \end{split}
 \end{align}
 and
 \begin{align}\label{eq:gii}
 \begin{split}
      g^h_{ii}&:=
 \begin{cases}\frac{1}{y_i-x_i}\cdot\Big(\frac{\barW^h_{i,i}-\barW^h_{i-1,i-1}}{y_i-hx_i}-\frac{\barW^h_{i-1,i}-\barW^h_{i-1,i-1}}{x_i-hx_i}\Big) & {\rm if\;} i\in I_h, \\
 \frac{1}{y_i-x_i}\cdot\Big(\frac{\barW^h_{i,i}-\barW^h_{i-1,i-1}}{y_i-hx_i}-\partial_i \barW^h_{i-1,i} \Big) & {\rm if\;} i\in I^h.
 \end{cases}
 \end{split}\end{align}
Observe that $g^h_{ji}=0$ if $j>i$, and $f^h_{ji}=0$ if $j \geq i$. Also, if $i\in I^h$ then $\frac{\barW^h_{i,i}-\barW^h_{i-1,i-1}}{y_i-hx_i}-\partial_i \barW^h_{i-1,i}$ is indeed divisible by $y_i-x_i$ because it vanishes after we evaluate $y_i=x_i$, so $g_{ii}^h$ is also well-defined.

For $S=k[x_1,\cdots,x_n,y_1,\cdots,y_n]$, define an $S$-linear map
\begin{align}
 \eta_h: S[\theta_1,\cdots,\theta_n,\partial_1,\cdots,\partial_n] & \to S[\theta_1,\cdots,\theta_n,\partial_1,\cdots,\partial_n],\nonumber \\
\theta_I \partial_J& \mapsto  \sum (-1)^{|I|} g^h_{ji} \frac{\partial\theta_I}{\partial\theta_i}\partial_j\partial_J+\sum f^h_{ji}\frac{\partial^2 \theta_I}{\partial\theta_j\partial\theta_i}\partial_J. \label{etasign}
\end{align}
We also define a map
\begin{equation}\label{eq:exp} 
\exp(\eta_h):=1+\eta_h+\frac{\eta_h^2}{2!}+\cdots\;\;:S[\theta_1,\cdots,\theta_n,\partial_1,\cdots,\partial_n]  \to S[\theta_1,\cdots,\theta_n,\partial_1,\cdots,\partial_n].
\end{equation}
\eqref{eq:exp} is a finite sum due to graded commutativity of $\theta$ and $\partial$, hence well-defined.

\begin{theorem}\label{thm:cocycle}
Let $h\in G$, and let $\displaystyle\sum_{|I|=r} a_I \theta_I\theta_{I_h}$ be a cocycle of the Koszul complex of $(\partial_i W^h : i\in I^h)$ where $a_I \in k[x_1^h,\cdots,x_n^h]$ is identified as an element of $S$ by natural inclusion. Then 
$\exp(\eta_h)\big(\sum a_I \theta_I\theta_{I_h}\big)$ is a closed element in $hom_{MF(W(y)-W(x))}(\Delta_1,\Delta_h)$.
\end{theorem}

\begin{proof}

Throughout the proof, we omit all upper indices $h$, i.e. 
\[g_{ji}=g^h_{ji},\;\; f_{ji}=f^h_{ji},\;\; \barW=\barW^h,\;\;\tildeW=\tildeW^h.\] 
%Let $I_h \subset I \subset \{1,\cdots,n\}$, and 
Let $f\theta_I \partial_J\in hom_{MF(W(y)-W(x))}(\Delta_1,\Delta_h)$. For the differential $D$ of the chain complex $hom_{MF(W(y)-W(x))}(\Delta_1,\Delta_h)$, decompose $ D=\Dup+\Ddown$ as follows:
\begin{align*} \Dup(f\theta_I\partial_J):=&\sum (y_i-hx_i)\theta_i \cdot f\theta_I\partial_J+(-1)^{|I|+|J|+1}\sum f\theta_I\partial_J \cdot (y_i-x_i)\theta_i,\\
 \Ddown(f\theta_I\partial_J):=&\sum \nabla_i W(hx,y)\partial_i\cdot f\theta_I\partial_J + (-1)^{|I|+|J|+1} \sum f\theta_I\partial_J \cdot \nabla_i W(x,y) \partial_i.
 \end{align*}
 
\begin{step} 
We have
\begin{align*}
 & \Dup \big(\sum_I a_I\theta_I\theta_{I_h}\big) \\
 =& \sum_{1\leq i \leq n}\sum_I (y_i-hx_i)\theta_i\cdot a_I\theta_I\theta_{I_h}
 + \sum_{1\leq i \leq n}\sum_I (-1)^{|I|+|I_h|+1}a_I\theta_I\theta_{I_h} \cdot (y_i-x_i)\theta_i \\
 =& \sum_{i\in I^h}\sum_I (y_i-hx_i)\theta_i\cdot a_I\theta_I\theta_{I_h}
 + \sum_{i\in I^h}\sum_I (-1)^{|I|+|I_h|+1}a_I\theta_I\theta_{I_h} \cdot (y_i-x_i)\theta_i  \\
 =& \sum_{i\in I^h}\sum_I (y_i-x_i)\theta_i \cdot a_I \theta_I \theta_{I_h}- \sum_{i\in I^h}\sum_I (y_i-x_i)\theta_i \cdot a_I \theta_I \theta_{I_h} \\
 =& 0.
\end{align*}
\end{step}

\begin{step}\label{step:dupdown}
%It suffices to prove
%\begin{equation}\label{eq:vanishing}
% \Ddown\Big(\frac{\eta_h^k}{k!}\big(\sum_I a_I \theta_I\theta_{I_h}\big)\Big)=-\Dup\Big(\frac{\eta_h^{k+1}}{(k+1)!}\big(\sum_I a_I \theta_I\theta_{I_h}\big)\Big)
% \end{equation}
%for any $k$.
Now we spell out $ \Ddown\Big( \frac{\eta_h^k}{k!}\big( a_I \theta_I\theta_{I_h}\big)\Big)$ and $ \Dup\Big(\frac{\eta_h^{k+1}}{(k+1)!}\big( a_I \theta_I\theta_{I_h}\big)\Big)$ for each $I\subset I^h$.

By graded commutativity (see \cite{LeeTwJac} for more detail), we have
\begin{equation}\label{eq:kthpower}
 \frac{\eta_h^k}{k!}( a_I \theta_I\theta_{I_h})=\sum_{\stackrel{l+m=k}{1\leq i_\bullet,j_\bullet\leq n}}  (-1)^{\epsilon_l}g_{j_1 i_1}\cdots g_{j_l i_l} f_{i_{l+1} i_{l+2}} \cdots f_{i_{l+2m-1} i_{l+2m}}\cdot  a_I \frac{\partial^{l+2m}(\theta_{I}\theta_{I_h})}{\partial\theta_{i_1}\cdots\partial\theta_{i_{l+2m}}} \partial_{j_1}\cdots \partial_{j_l},
\end{equation}
with the sign given by 
$\epsilon_l:=l(|I|+|I_h|)-\frac{l(l-1)}{2}.$
The following identities will be frequently (and implicitly) used in the sequel:
\[ (-1)^{|I|+|I_h|+\epsilon_l+l}=(-1)^{\epsilon_{l+1}},\quad (-1)^{\epsilon_{l+2}}=(-1)^{\epsilon_l+1}.\]
We introduce a notation
\[ \Phi_{(j_1,i_1),\cdots,(j_l,i_l)}^{(i_{l+1},i_{l+2}),\cdots,(i_{l+2m-1},i_{l+2m})}
:=g_{j_1 i_1}\cdots g_{j_l i_l} f_{i_{l+1} i_{l+2}} \cdots f_{i_{l+2m-1} i_{l+2m}}.\]
%\cdot  a_I \frac{\partial^{l+2m}(\theta_{I}\theta_{I_h})}{\partial\theta_{i_1}\cdots\partial\theta_{i_{l+2m}}} \partial_{j_1}\cdots \partial_{j_l},\]
%%
%so that
%%
%\[ \frac{\eta_h^k}{k!}\big(\sum_I a_I \theta_I\theta_{I_h}\big)
%=\sum_{\stackrel{l+m=k}{1\leq i_\bullet,j_\bullet\leq n}} \Phi_{(j_1,i_1),\cdots,(j_l,i_l)}^{(i_{l+1},i_{l+2}),\cdots,(i_{l+2m-1},i_{l+2m})}.
%\]
%
Then we can write
\begin{align*}
& \Ddown\Big( \frac{\eta_h^k}{k!}\big( a_I \theta_I\theta_{I_h}\big)\Big)  \\
=& \sum_{\stackrel{1\leq i \leq n}{l+m=k,1\leq i_\bullet,j_\bullet\leq n}}\nabla_i W(hx,y)\partial_i\cdot (-1)^{\epsilon_l}\Phi_{(j_1,i_1),\cdots,(j_l,i_l)}^{(i_{l+1},i_{l+2}),\cdots,(i_{l+2m-1},i_{l+2m})}\cdot a_I \frac{\partial^{l+2m} (\theta_I\theta_{I_h})}{\partial\theta_{i_1}\cdots\partial\theta_{i_{l+2m}}} \partial_{j_1}\cdots \partial_{j_l} \nonumber \\
&+  \sum_{\stackrel{1\leq j \leq n}{l+m=k,1\leq i_\bullet,j_\bullet\leq n}}(-1)^{|I|+|I_h|+1+\epsilon_l}
\Phi_{(j_1,i_1),\cdots,(j_l,i_l)}^{(i_{l+1},i_{l+2}),\cdots,(i_{l+2m-1},i_{l+2m})}
\cdot a_I \frac{\partial^{l+2m} (\theta_I\theta_{I_h})}{\partial\theta_{i_1}\cdots\partial\theta_{i_{l+2m}}} \partial_{j_1}\cdots \partial_{j_l}\cdot(\nabla_j W) \partial_j \nonumber\\
& \nonumber \\
=& \sum_{\stackrel{1\leq j \leq n}{l+m=k,1\leq i_\bullet,j_\bullet\leq n}} (-1)^{|I|+|I_h|+\epsilon_l} (\nabla_j W(hx,y)-\nabla_j W)
\cdot \Phi_{(j_1,i_1),\cdots,(j_l,i_l)}^{(i_{l+1},i_{l+2}),\cdots,(i_{l+2m-1},i_{l+2m})}
\cdot a_I \frac{\partial^{l+2m} (\theta_I\theta_{I_h})}{\partial\theta_{i_1}\cdots\partial\theta_{i_{l+2m}}} \partial_{j_1}\cdots \partial_{j_l}\partial_j \nonumber \\
&+\sum_{\stackrel{1\leq i \leq n}{l+m=k,1\leq i_\bullet,j_\bullet\leq n}} (-1)^{\epsilon_l} \nabla_i W(hx,y)\cdot 
\Phi_{(j_1,i_1),\cdots,(j_l,i_l)}^{(i_{l+1},i_{l+2}),\cdots,(i_{l+2m-1},i_{l+2m})}
\cdot a_I\frac{\partial^{l+2m+1}(\theta_{I}\theta_{I_h})}{\partial\theta_i\partial\theta_{i_1}\cdots\partial\theta_{i_{l+2m}}} \partial_{j_1}\cdots \partial_{j_l} \nonumber \\
&\nonumber\\
=& \sum_{\stackrel{1\leq j \leq n}{l+m=k,1\leq i_\bullet,j_\bullet\leq n}} (-1)^{|I|+|I_h|+\epsilon_l+l} (\nabla_j W(hx,y)-\nabla_j W)
\cdot\Phi_{(j_1,i_1),\cdots,(j_l,i_l)}^{(i_{l+1},i_{l+2}),\cdots,(i_{l+2m-1},i_{l+2m})}
\cdot a_I \frac{\partial^{l+2m} (\theta_I\theta_{I_h})}{\partial\theta_{i_1}\cdots\partial\theta_{i_{l+2m}}} \partial_j\partial_{j_1}\cdots \partial_{j_l} \\
&+\sum_{\stackrel{1\leq i \leq n}{l+m=k,1\leq i_\bullet,j_\bullet\leq n}} (-1)^{\epsilon_l} \nabla_i W(hx,y)
\cdot \Phi_{(j_1,i_1),\cdots,(j_l,i_l)}^{(i_{l+1},i_{l+2}),\cdots,(i_{l+2m-1},i_{l+2m})}
\cdot a_I\frac{\partial^{l+2m+1}(\theta_{I}\theta_{I_h})}{\partial\theta_i\partial\theta_{i_1}\cdots\partial\theta_{i_{l+2m}}} \partial_{j_1}\cdots \partial_{j_l} \\ 
&\nonumber\\
=& \sum_{\stackrel{1\leq j \leq n}{l+m=k,1\leq i_\bullet,j_\bullet\leq n}} (-1)^{\epsilon_{l+1}} (\nabla_j W(hx,y)-\nabla_j W)
\cdot\Phi_{(j_1,i_1),\cdots,(j_l,i_l)}^{(i_{l+1},i_{l+2}),\cdots,(i_{l+2m-1},i_{l+2m})}
\cdot a_I \frac{\partial^{l+2m} (\theta_I\theta_{I_h})}{\partial\theta_{i_1}\cdots\partial\theta_{i_{l+2m}}} \partial_j\partial_{j_1}\cdots \partial_{j_l}\\
&+\sum_{\stackrel{1\leq i \leq n}{l+m=k,1\leq i_\bullet,j_\bullet\leq n}} (-1)^{\epsilon_l} \nabla_i W(hx,y)
\cdot \Phi_{(j_1,i_1),\cdots,(j_l,i_l)}^{(i_{l+1},i_{l+2}),\cdots,(i_{l+2m-1},i_{l+2m})}
\cdot a_I\frac{\partial^{l+2m+1}(\theta_{I}\theta_{I_h})}{\partial\theta_i\partial\theta_{i_1}\cdots\partial\theta_{i_{l+2m}}} \partial_{j_1}\cdots \partial_{j_l} \nonumber \\
=& \sum_{\stackrel{1\leq j \leq n}{l+m=k,1\leq i_\bullet,j_\bullet\leq n}} (-1)^{\epsilon_{l+1}} (\nabla_j W(hx,y)-\nabla_j W)
\cdot\Phi_{(j_1,i_1),\cdots,(j_l,i_l)}^{(i_{l+1},i_{l+2}),\cdots,(i_{l+2m-1},i_{l+2m})}
\cdot a_I \frac{\partial^{l+2m} (\theta_I\theta_{I_h})}{\partial\theta_{i_1}\cdots\partial\theta_{i_{l+2m}}} \partial_j\partial_{j_1}\cdots \partial_{j_l} \nonumber \\
& +\sum_{\stackrel{1\leq i,j \leq n}{l+m=k,1\leq i_\bullet,j_\bullet\leq n}} (-1)^{\epsilon_{l+1}} \nabla_i W(hx,y)
\cdot \Phi_{(j,i_{l+2}),(j_1,i_1),\cdots,(j_l,i_l)}^{(i_{l+3},i_{l+4}),\cdots,(i_{l+2m-1},i_{l+2m})} \\
&\qquad\qquad\qquad \qquad
\cdot a_I\frac{\partial^{l+2m}(\theta_{I}\theta_{I_h})}{\partial\theta_i\partial\theta_{i_{l+2}}\partial\theta_{i_1}\cdots\partial\theta_{i_l}\partial\theta_{i_{l+3}}\cdots\partial\theta_{i_{l+2m}}} 
\partial_j\partial_{j_1}\cdots \partial_{j_l}. \nonumber
\end{align*}
For $K=\{k_1,\cdots,k_i\}$ and $L=\{l_1,\cdots,l_j\}$ which are ordered subsets of $\{1,\cdots,n\}$, we denote
\[ \theta_K:=\theta_{k_1}\cdots\theta_{k_i},\quad \partial_L:=\partial_{l_1}\cdots\partial_{l_j}.\]
For
$\Psi=\sum\limits_{K,L} a_{KL} \theta_K \partial_L \in S[\theta_1,\cdots,\theta_n,\partial_1,\cdots,\partial_n],$ denote the summand $a_{KL} \theta_K \partial_L$ by $\langle \Psi,\theta_K\partial_L\rangle$. Be careful that $\langle \Psi,\theta_K\partial_L\rangle$ is not just $a_{KL}$.

So we compute $\langle \Ddown\Big( \frac{\eta_h^k}{k!}( a_I \theta_I\theta_{I_h})\Big), 
\frac{\partial^{l+2m} (\theta_I\theta_{I_h})}{\partial\theta_{i_1}\cdots\partial\theta_{i_{l+2m}}} \partial_{j_0}\partial_{j_1}\cdots \partial_{j_l} \rangle$ for fixed indices $i_1,\cdots,i_{l+2m}$ and $j_0,\cdots,j_l$, where $l+m=k$. We have
\begin{align}
& \Big\langle \Ddown\Big( \frac{\eta_h^k}{k!}( a_I \theta_I\theta_{I_h})\Big), 
\frac{\partial^{l+2m} (\theta_I\theta_{I_h})}{\partial\theta_{i_1}\cdots\partial\theta_{i_{l+2m}}} \partial_{j_0}\partial_{j_1}\cdots \partial_{j_l} \Big\rangle \label{eq:ddown1} \\
=& \sum_{\stackrel{\sigma\in {\textrm{Perm}}(\{j_0,\cdots,j_l\})}{\tau\in {\textrm{Perm}}(\{i_1,\cdots,i_{l+2m}\})}}
(-1)^{\epsilon_{l+1}} \Big(  (\nabla_{\sigma(j_0)} W(hx,y)-\nabla_{\sigma(j_0)} W)
\cdot \Phi_{(\sigma(j_1),\tau( i_1)),\cdots,(\sigma(j_l), \tau(i_l))}^{(\tau(i_{l+1}), \tau(i_{l+2})),\cdots,(\tau(i_{l+2m-1}),\tau (i_{l+2m}))}\nonumber \\
&\qquad\qquad\qquad\qquad\qquad \qquad
+ \nabla_{\tau(i_{l+1})}W(hx,y)
\cdot \Phi_{(\sigma(j_0),\tau( i_{l+2})),(\sigma(j_1),\tau( i_1)),\cdots,(\sigma(j_l), \tau(i_l))}^{(\tau(i_{l+3}), \tau(i_{l+4})),\cdots,(\tau(i_{l+2m-1}),\tau (i_{l+2m}))} \Big) \nonumber \\
&\qquad\qquad\qquad\qquad\qquad\qquad
\cdot a_I \frac{\partial^{l+2m} (\theta_I\theta_{I_h})}{\partial\theta_{\tau(i_1)}\cdots\partial\theta_{\tau(i_{l+2m})}}\partial_{\sigma(j_0)}\partial_{\sigma(j_1)}\cdots \partial_{\sigma(j_l)} \nonumber
\end{align}
%
%and
%%
%\begin{align}
%& \langle \Ddown\Big( \frac{\eta_h^k}{k!}( a_I \theta_I\theta_{I_h})\Big), 
%\frac{\partial^{l+2m+1} (\theta_I\theta_{I_h})}
%{\partial\theta_{i_0} \partial\theta_{i_1}\cdots\partial\theta_{i_{l+2m}}} \partial_{j_1}\cdots \partial_{j_l} \rangle \label{eq:ddown2} \\
%%
%=&\sum_{\stackrel{\sigma\in {\textrm{Perm}}(\{j_1,\cdots,j_l\})}{\tau\in {\textrm{Perm}}(\{i_0,\cdots,i_{l+2m}\})}}
%\Big((-1)^{\epsilon_l} \nabla_{\tau(i_0)} W(hx,y)
%\cdot \Phi_{(\sigma(j_1),\tau( i_1)),\cdots,(\sigma(j_l), \tau(i_l))}^{(\tau(i_{l+1}), \tau(i_{l+2})),\cdots,(\tau(i_{l+2m-1}),\tau (i_{l+2m}))}\nonumber \\
%&\qquad\qquad\qquad\qquad 
%+ (-1)^{\epsilon_{l}} (\nabla_{\sigma(j_1)}W(hx,y)-\nabla_{\sigma(j_1)}W)
%\cdot \Phi_{(\sigma(j_2),\tau( i_2)),\cdots,(\sigma(j_l), \tau(i_l))}^{(\tau(i_0),\tau(i_1)),(\tau(i_{l+1}), \tau(i_{l+2})),\cdots,(\tau(i_{l+2m-1}),\tau (i_{l+2m}))}\Big) \nonumber \\
%&\qquad \qquad\qquad\qquad \cdot a_I \frac{\partial^{l+2m+1} (\theta_I\theta_{I_h})}{\partial\theta_{\tau(i_0)}\partial\theta_{\tau(i_1)}\cdots\partial\theta_{\tau(i_{l+2m})}} \partial_{\sigma(j_1)}\cdots \partial_{\sigma(j_l)} \nonumber
%\end{align}
%
where ${\textrm{Perm}}(A)$ is the set of permutations of $A$.

%By definition, 
% \[ \Dup\Big(\frac{\eta_h^{k+1}}{(k+1)!}\theta_{I_h}\Big)=\sum (y_i-hx_i)\theta_i \cdot \frac{\eta_h^{k+1}}{(k+1)!}\theta_{I_h}
%+ (-1)^{|I_h|+1}\sum\frac{\eta_h^{k+1}}{(k+1)!}\theta_{I_h}\cdot (y_i-x_i)\theta_i.\]
%To compute $\Dup\Big(\frac{\eta_h^{k+1}}{(k+1)!}\big(\sum_I a_I \theta_I\theta_{I_h}\big)\Big)$, we introduce another notation
%%
%\begin{align*}
%&\Psi_{(j_1,i_1),\cdots,(j_l,i_l)}^{(i_{l+1},i_{l+2}),\cdots,(i_{l+2m-1},i_{l+2m})}\\
%:=&\sum_{i,j}\sum_I\big((-1)^{\epsilon_{l+1}}g_{ji}g_{j_1 i_1}\cdots g_{j_l i_l} f_{i_{l+1} i_{l+2}} \cdots f_{i_{l+2m-1} i_{l+2m}}\cdot a_I\frac{\partial^{l+2m+1}(\theta_I\theta_{I_h})}{\partial\theta_i\partial\theta_{i_1}\cdots\partial\theta_{i_{l+2m}}} \partial_j\partial_{j_1}\cdots \partial_{j_l} \\
%&\quad +(-1)^{\epsilon_{l}}f_{ji}g_{j_1 i_1}\cdots g_{j_l i_l} f_{i_{l+1} i_{l+2}} \cdots f_{i_{l+2m-1} i_{l+2m}}\cdot a_I\frac{\partial^{l+2m+2}(\theta_I\theta_{I_h})}{\partial\theta_j\partial\theta_i\partial\theta_{i_1}\cdots\partial\theta_{i_{l+2m}}} \partial_{j_1}\cdots \partial_{j_l}\big),
%\end{align*}
%%
%so that
%\[ \frac{\eta_h^{k+1}}{(k+1)!}\big(\sum_I a_I \theta_I\theta_{I_h}\big)= \sum_{\stackrel{l+m=k}{1\leq i_\bullet,j_\bullet\leq n}}
%\Psi_{(j_1,i_1),\cdots,(j_l,i_l)}^{(i_{l+1},i_{l+2}),\cdots,(i_{l+2m-1},i_{l+2m})}.\]
% %{\color{red}(?) \;seems \;true}\]

Next, we will compute
$\langle
\Dup\Big(\frac{\eta_h^{k+1}}{(k+1)!}\big( a_I \theta_I\theta_{I_h}\big)\Big),
\frac{\partial^{l+2m}(\theta_I\theta_{I_h})}{\partial\theta_{i_1}\cdots\partial\theta_{i_{l+2m}}} \partial_{j_0}\partial_{j_1}\cdots \partial_{j_l}
\rangle$
accordingly.
%and
%%
%\[\langle 
%\Dup\Big(\frac{\eta_h^{k+1}}{(k+1)!}( a_I \theta_I\theta_{I_h})\Big),
%\frac{\partial^{l+2m+1}(\theta_I\theta_{I_h})}{\partial\theta_{i_0}\partial\theta_{i_1}\cdots\partial\theta_{i_{l+2m}}} 
%\partial_{j_1}\cdots \partial_{j_l} 
%\rangle.\]
%%
%
%
%To compute $\langle\Dup\Big(\frac{\eta_h^{k+1}}{(k+1)!}\big(\sum_I a_I \theta_I\theta_{I_h}\big)\Big),\frac{\partial^{l+2m}(\theta_I\theta_{I_h})}{\partial\theta_{i_1}\cdots\partial\theta_{i_{l+2m}}} \partial_j\partial_{j_1}\cdots \partial_{j_l} \rangle$, it suffices to consider the following types:
\begin{align*}
& \Dup\Big(\frac{\eta_h^{k+1}}{(k+1)!}\big( a_I \theta_I\theta_{I_h}\big)\Big) \\
=& \sum_{\stackrel{1\leq j \leq n}{l+m=k,1\leq i_\bullet,j_\bullet\leq n}}(-1)^{\epsilon_{l+1}} (y_i-hx_i)\theta_i \cdot
 g_{j_0i_0} \Phi_{(j_1 , i_1),\cdots,(j_l,i_l)}^{(i_{l+1},i_{l+2}),\cdots,(i_{l+2m-1},i_{l+2m})}
\cdot a_I \frac{\partial^{l+2m+1}(\theta_I\theta_{I_h})}{\partial\theta_{i_0}\partial\theta_{i_1}\cdots\partial\theta_{i_{l+2m}}} 
\partial_{j_0}\partial_{j_1}\cdots \partial_{j_l}\nonumber\\
&+ \sum_{\stackrel{1\leq j \leq n}{l+m=k,1\leq i_\bullet,j_\bullet\leq n}}  (-1)^{\epsilon_{l+2}+|I|+|I_h|} (y_{i}-x_{i})
g_{r i_{l+1}}\Phi_{(j_0 , i_{i_{l+2}}),(j_1,i_1),\cdots,(j_l,i_l)}^{(i_{l+3},i_{l+4}),\cdots,(i_{l+2m-1},i_{l+2m})}\nonumber \\
&\qquad\qquad\cdot a_I \frac{\partial^{l+2m}(\theta_I\theta_{I_h})}{\partial\theta_{i_{l+1}}\partial\theta_{i_{l+2}}\partial\theta_{i_1}\cdots\partial\theta_{i_l}\partial\theta_{i_{l+3}}\cdots\partial\theta_{i_{l+2m}}} 
(\partial_r \partial_{j_0}\partial_{j_1}\cdots \partial_{j_l})\cdot \theta_i, \nonumber
%& \nonumber\\
%
%=& \sum_{ \stackrel{1\leq i \leq n,}{i\neq i_1,\cdots,i_{l+2m}}}\sum_I (x_i-hx_i)\cdot (-1)^{\epsilon_{l+1}} g_{j_0 i_0}g_{j_1 i_1}\cdots g_{j_l i_l} f_{i_{l+1} i_{l+2}} \cdots f_{i_{l+2m-1} i_{l+2m}}\cdot 
%a_I\frac{\partial^{l+2m}(\theta_I\theta_{I_h})}{\partial\theta_{i_1}\cdots\partial\theta_{i_{l+2m}}} 
%\partial_{j_0}\partial_{j_1}\cdots \partial_{j_l} \nonumber
\end{align*}
hence
\begin{align}
& \Big\langle \Dup\Big(\frac{\eta_h^{k+1}}{(k+1)!}\big( a_I \theta_I\theta_{I_h}\big)\Big),
\frac{\partial^{l+2m}(\theta_I\theta_{I_h})}{\partial\theta_{i_1}\cdots\partial\theta_{i_{l+2m}}} 
\partial_{j_0}\partial_{j_1}\cdots \partial_{j_l} \Big\rangle \label{eq:dup2}\\
& \nonumber\\
=& \sum_{\stackrel{\sigma\in {\textrm{Perm}} (\{j_0,\cdots,j_l\})}{\tau\in {\textrm{Perm}}(\{i_1,\cdots,i_{l+2m}\})}}
\Big(\sum_{1\leq i \leq n} (-1)^{\epsilon_{l+1}}(x_{i}-hx_{i}) g_{\sigma(j_0), i} 
\cdot \Phi_{(\sigma(j_1),\tau( i_1)),\cdots,(\sigma(j_l), \tau(i_l))}^{(\tau(i_{l+1}), \tau(i_{l+2})),\cdots,(\tau(i_{l+2m-1}),\tau (i_{l+2m}))} \nonumber\\
 & \qquad\qquad\qquad\qquad\qquad
  \cdot a_I \theta_i\frac{\partial^{l+2m+1} (\theta_I\theta_{I_h})}{\partial\theta_i\partial\theta_{\tau(i_1)}\cdots\partial\theta_{\tau(i_{l+2m})}}\partial_{\sigma(j_0)}\partial_{\sigma(j_1)}\cdots \partial_{\sigma(j_l)} \nonumber\\
&\qquad\qquad\qquad\qquad +\sum_{j>\tau(i_{l+1})} (-1)^{\epsilon_{l+1}} 
(x_j-hx_j)f_{\tau(i_{l+1}),j}
\cdot  \Phi_{(\sigma(j_0),\tau( i_{l+2})),(\sigma(j_1,\tau(i_1)),\cdots,(\sigma(j_l), \tau(i_l))}^{(\tau(i_{l+3}), \tau(i_{l+4})),\cdots,(\tau(i_{l+2m-1}),\tau (i_{l+2m}))}\nonumber \\
&\qquad\qquad\qquad\qquad \qquad
\cdot a_I\theta_j \frac{\partial^{l+2m+1} (\theta_I\theta_{I_h})}
{\partial\theta_{\tau(i_{l+1})}\partial\theta_{j}\partial\theta_{\tau(i_{l+2})}
\partial\theta_{\tau(i_1)}\cdots\partial\theta_{\tau(i_l)}\partial\theta_{\tau(i_{l+3})}\cdots\partial\theta_{\tau(i_{l+2m})}}
\partial_{\sigma(j_0)}\partial_{\sigma(j_1)}\cdots \partial_{\sigma(j_l)} \nonumber\\
& \nonumber\\
&\qquad\qquad\qquad\qquad +\sum_{k<\tau(i_{l+1})} (-1)^{\epsilon_{l+1}} 
(x_k-hx_k)f_{k,\tau(i_{l+1})}
\cdot  \Phi_{(\sigma(j_0),\tau( i_{l+2})),(\sigma(j_1,\tau(i_1)),\cdots,(\sigma(j_l), \tau(i_l))}^{(\tau(i_{l+3}), \tau(i_{l+4})),\cdots,(\tau(i_{l+2m-1}),\tau (i_{l+2m}))}\nonumber \\
&\qquad\qquad\qquad\qquad \qquad
\cdot a_I\theta_k \frac{\partial^{l+2m+1} (\theta_I\theta_{I_h})}
{\partial\theta_k\partial\theta_{\tau(i_{l+1})}\partial\theta_{\tau(i_{l+2})}
\partial\theta_{\tau(i_1)}\cdots\partial\theta_{\tau(i_l)}\partial\theta_{\tau(i_{l+3})}\cdots\partial\theta_{\tau(i_{l+2m})}}
\partial_{\sigma(j_0)}\partial_{\sigma(j_1)}\cdots \partial_{\sigma(j_l)} \nonumber\\
& \nonumber\\
&\qquad\qquad\qquad\qquad +\sum_{r\neq j_0,\cdots,j_n} (-1)^{\epsilon_{l+2}+|I|+|I_h|+l+2} 
(y_{r}-x_{r})g_{ r ,\tau( i_{l+1})}
\cdot  \Phi_{(\sigma(j_0),\tau( i_{l+2})),(\sigma(j_1,\tau(i_1)),\cdots,(\sigma(j_l), \tau(i_l))}^{(\tau(i_{l+3}), \tau(i_{l+4})),\cdots,(\tau(i_{l+2m-1}),\tau (i_{l+2m}))}\nonumber \\
&\qquad\qquad\qquad\qquad\qquad
 \cdot a_I \frac{\partial^{l+2m} (\theta_I\theta_{I_h})}
{\partial\theta_{\tau(i_{l+1})}\partial\theta_{\tau(i_{l+2})}
\partial\theta_{\tau(i_1)}\cdots\partial\theta_{\tau(i_l)}\partial\theta_{\tau(i_{l+3})}\cdots\partial\theta_{\tau(i_{l+2m})}}
\Big(\frac{\partial\theta_r}{\partial\theta_r}\Big) \partial_{\sigma(j_0)}\partial_{\sigma(j_1)}\cdots \partial_{\sigma(j_l)}\Big)  \nonumber\\
& \nonumber\\
=& \sum_{\stackrel{\sigma\in {\textrm{Perm}} (\{j_0,\cdots,j_l\})}{\tau\in {\textrm{Perm}}(\{i_1,\cdots,i_{l+2m}\})}}
\Big[ \sum_{i\neq i_1,\cdots,i_{l+2m}} (-1)^{\epsilon_{l+1}}(x_{i}-hx_{i}) g_{\sigma(j_0), i} 
\cdot \Phi_{(\sigma(j_1),\tau( i_1)),\cdots,(\sigma(j_l), \tau(i_l))}^{(\tau(i_{l+1}), \tau(i_{l+2})),\cdots,(\tau(i_{l+2m-1}),\tau (i_{l+2m}))}  \nonumber\\
&\qquad\qquad\qquad\qquad
 +\Big( \sum_{r\neq j_0,\cdots,j_l} (-1)^{\epsilon_{l+3}} (y_{r}-x_{r})
g_{ r ,\tau( i_{l+1})} 
+ \sum_{j \neq i_1,\cdots,i_{l+2m}} (-1)^{\epsilon_{l+1}+1} (x_j-hx_j)f_{\tau(i_{l+1}),j} \nonumber\\
&\qquad\qquad\qquad\qquad\qquad
+ \sum_{k \neq i_1,\cdots,i_{l+2m}} (-1)^{\epsilon_{l+1}} (x_k-hx_k)f_{k,\tau(i_{l+1})} \Big)
\cdot \Phi_{(\sigma(j_0),\tau( i_2)),(\sigma(j_1),\tau(i_1)),\cdots,(\sigma(j_l), \tau(i_l))}^{(\tau(i_{l+3}), \tau(i_{l+4})),\cdots,(\tau(i_{l+2m-1}),\tau (i_{l+2m}))} \Big] \nonumber\\
&\qquad\qquad\qquad\qquad
\cdot a_I \frac{\partial^{l+2m} (\theta_I\theta_{I_h})}{\partial\theta_{\tau(i_1)}\cdots\partial\theta_{\tau(i_{l+2m})}} \partial_{\sigma(j_0)}\partial_{\sigma(j_1)}\cdots \partial_{\sigma(j_l)}. \nonumber
\end{align}
At the end of the computation, we needed to exclude some summands. Nevertheless, it can be shown that if we add those "dummy" terms to \eqref{eq:dup2}, then the additional dummy terms does not affect the whole computation, because of the cancellation in pairs among newly added dummy terms (we refer readers to \cite{LeeTwJac} for the detail). In summary, we have
\begin{align}
& \Big\langle \Dup\Big(\frac{\eta_h^{k+1}}{(k+1)!}\big( a_I \theta_I\theta_{I_h}\big)\Big),
\frac{\partial^{l+2m}(\theta_I\theta_{I_h})}{\partial\theta_{i_1}\cdots\partial\theta_{i_{l+2m}}} 
\partial_{j_0}\partial_{j_1}\cdots \partial_{j_l} \Big\rangle \label{eq:dup3}\\
& \nonumber\\
=& \sum_{\stackrel{\sigma\in {\textrm{Perm}} (\{j_0,\cdots,j_l\})}{\tau\in {\textrm{Perm}}(\{i_1,\cdots,i_{l+2m}\})}}
(-1)^{\epsilon_{l+1}} \Big[ \sum_{i\geq \sigma(j_0)}(x_{i}-hx_{i}) g_{\sigma(j_0), i} 
\cdot \Phi_{(\sigma(j_1),\tau( i_1)),\cdots,(\sigma(j_l), \tau(i_l))}^{(\tau(i_{l+1}), \tau(i_{l+2})),\cdots,(\tau(i_{l+2m-1}),\tau (i_{l+2m}))}  \nonumber\\
&\qquad\qquad\qquad\qquad\qquad\qquad
 +\Big( -\sum_{r\leq \tau(i_{l+1})}  (y_{r}-x_{r})
g_{ r ,\tau( i_{l+1})} 
- \sum_{j> \tau(i_{l+1})} (x_j-hx_j)f_{\tau(i_{l+1}),j} \nonumber\\
&\qquad\qquad\qquad\qquad\qquad\qquad\qquad
+ \sum_{k<\tau(i_{l+1})} (x_k-hx_k)f_{k,\tau(i_{l+1})} \Big)
\cdot \Phi_{(\sigma(j_0),\tau( i_{l+2})),(\sigma(j_1),\tau(i_1)),\cdots,(\sigma(j_l), \tau(i_l))}^{(\tau(i_{l+3}), \tau(i_{l+4})),\cdots,(\tau(i_{l+2m-1}),\tau (i_{l+2m}))} \Big] \nonumber\\
&\qquad\qquad\qquad\qquad\qquad\qquad
\cdot a_I \frac{\partial^{l+2m} (\theta_I\theta_{I_h})}{\partial\theta_{\tau(i_1)}\cdots\partial\theta_{\tau(i_{l+2m})}} \partial_{\sigma(j_0)}\partial_{\sigma(j_1)}\cdots \partial_{\sigma(j_l)}. \nonumber
\end{align}
\end{step}

\begin{step}
We show
\begin{align}
\sum_{i: i\geq j} (x_i-hx_i)g_{ji}&=\nabla_j W - \nabla_j W(hx,y) \label{eq:nabladiff}
\end{align}
and
\begin{align}\label{eq:nablahx}
\sum_{j: j\leq i}(y_j-x_j) g_{ji}+ \sum_{j:j>i} (x_j-hx_j) f_{ij}-\sum_{k:k<i} (x_k-hx_k) f_{ki}&= \nabla_i W(hx,y)-\partial_i W^h. 
\end{align}
(By definition of $W^h$, $\partial_i W^h=0$ if $i\notin I^h$.)

If $j\in I_h$, then by direct computation we have
\[ (x_j-hx_j)g_{jj}=-\frac{\barW_{j,j}-\barW_{j-1,j-1}}{y_j-hx_j}+\frac{\barW_{j,j}-\barW_{j-1,j}}{y_j-x_j}.\]
It still holds for $j\in I^h$, namely the right hand side is zero, because in this case 
\[ \frac{\barW_{j,j}-\barW_{j-1,j-1}}{y_j-hx_j}=\frac{\barW_{j,j}-\barW_{j-1,j}}{y_j-x_j}.\]
For $j<i\leq n$, we can show
\[(x_i-hx_i)g_{ji}= \frac{(\barW_{j,i}-\barW_{j-1,i})-(\barW_{j,i-1}-\barW_{j-1,i-1})}{y_j-x_j}\]
as follows: for $i\in I_h$, it is straightforward by definition of $g_{ji}$. If $i\in I^h$, then the LHS is zero, and the RHS is also zero due to $\barW_{j,i}=\barW_{j,i-1}$ and $\barW_{j-1,i}=\barW_{j-1,i-1}$. Hence
\[\sum_{i: i \geq j}(x_i-hx_i) g_{ji} = 
-\frac{\barW_{j,j}-\barW_{j-1,j-1}}{y_j-hx_j}+\frac{\barW_{j,n}-\barW_{j-1,n}}{y_j-x_j}.\]
Now it is straightforward from \eqref{eq:Wbar} that 
\[  \frac{\barW_{j,j}-\barW_{j-1,j-1}}{y_j-hx_j}= \nabla_j W(hx,y),\quad 
\frac{\barW_{j,n}-\barW_{j-1,n}}{y_j-x_j}= \nabla_j W\]
and \eqref{eq:nabladiff} is proved.

%\begin{align*}
%&\langle \Dup\big(\eta_h (f\theta_I\partial_J)\big),\theta_{I-\{i\}}\partial_J\rangle \\
%=&(-1)^{|I|+|J|+1} \sum_{j=1}^i (y_j-x_j)(-1)^{|I|+|J|}g_{ji}f \frac{\partial\theta_I}{\partial\theta_i}\Big(\frac{\partial \theta_j}{\partial \theta_j}\Big)\partial_J \\
%&+\Big(\sum_{k=1}^{i-1} (x_k-hx_k) f_{ki}f \cdot \theta_k \frac{\partial^2\theta_I}{\partial\theta_k\partial\theta_i}\partial_J
%+\sum_{j=i+1}^n (x_j-hx_j) f_{ij} f \cdot\theta_j \frac{\partial^2 \theta_I}{\partial\theta_i \partial\theta_j}\partial_J\Big) \\
%=& -\sum_{j=1}^i  (y_j-x_j) g_{ji} f\frac{\partial\theta_I}{\partial\theta_i}\partial_J 
%+ \sum_{k=1}^{i-1}(x_k-hx_k) f_{ki}f\frac{\partial\theta_I}{\partial\theta_i}\partial_J -\sum_{j=i+1}^n  (x_j-hx_j) f_{ij}f \frac{\partial\theta_I}{\partial\theta_i}.
%\end{align*}
It remains to show \eqref{eq:nablahx}. First we assume $i \in I_h$. By direct computation we have
\[ -\sum_{j =1}^i(y_j-x_j) g_{ji}=  -\nabla_i W(hx,y)+\frac{\barW_{0,i}-\barW_{0,i-1}}{x_i-hx_i}\]
and
\begin{align}
& \sum_{k=1}^{i-1} (x_k-hx_k) f_{ki} -\sum_{j=i+1}^n (x_j-hx_j) f_{ij}\nonumber  \\
=&\sum_{k=1}^{i-1}\frac{(\tildeW_{k,i}-\tildeW_{k-1,i})-(\tildeW_{k,i-1}-\tildeW_{k-1,i-1})}{x_i-hx_i}
- \sum_{j=i+1}^n\frac{(\tildeW_{i,j}-\tildeW_{i-1,j})-(\tildeW_{i,j-1}-\tildeW_{i-1,j-1})}{x_i-hx_i}\nonumber\\
=&\frac{(\tildeW_{i-1,i}-\tildeW_{0,i})-(\tildeW_{i-1,i-1}-\tildeW_{0,i-1})+(\tildeW_{i-1,n}-\tildeW_{i-1,i})-(\tildeW_{i,n}-\tildeW_{i,i})}{x_i-hx_i} \nonumber \\
=& \frac{\tildeW_{0,i-1}-\tildeW_{0,i}-\tildeW_{i-1,i-1}+\tildeW_{i-1,n}-\tildeW_{i,n}+\tildeW_{i,i}}{x_i-hx_i} \nonumber \\
=&- \frac{\barW_{0,i}-\barW_{0,i-1}}{x_i-hx_i}.
\end{align}
The last identity is due to $\barW_{0,i}=\tildeW_{0,i}$ and $\tildeW_{i,i}=\tildeW_{i,n}$. Therefore, if $i\in I_h$, then
\[\sum_{j=1}^n (y_j-x_j) g_{ji}+ \sum_{j:j>i} (x_j-hx_j) f_{ij}-\sum_{k:k<i} (x_k-hx_k) f_{ki} = \nabla_i W(hx,y).\]
Next let $i\in I^h$. We have 
\[ -\sum_{j=1}^i (y_j-x_j)g_{ji}= -\nabla_i W(hx,y)+\partial_i \big(W(x_1,\cdots,x_i,hx_{i+1},\cdots,hx_n)\big),\]
and
\begin{align*}
& \sum_{k=1}^{i-1} (x_k-hx_k) f_{ki} -\sum_{j=i+1}^n (x_j-hx_j) f_{ij} \\
=& \partial_i \big(W(x_1^h,\cdots,x_{i-1}^h,x_i,hx_{i+1},\cdots,hx_n)\big) - \partial_i \big( W (x_1,\cdots,x_i,hx_{i+1},\cdots,hx_n) \big)\\
&+ \partial_i \big(W(x_1^h,\cdots,x_i^h,hx_{i+1},\cdots,hx_n)\big)-\partial_i \big(W(x_1^h,\cdots,x_i^h,x_{i+1},\cdots,x_n)\big)\\
&+ \partial_i \big(W(x_1^h,\cdots,x_n^h)\big)-\partial_i \big( W(x_1^h,\cdots,x_i^h,hx_{i+1},\cdots,hx_n )\big)\\
=& \partial_i W^h - \partial_i \big(W(x_1,\cdots,x_i,hx_{i+1},\cdots,hx_n)\big)
\end{align*}
where the last identity is due to $\tildeW_{i,i}=\tildeW_{i,n}$ and $x_i=hx_i$. We conclude that for $i\in I^h$
\[ \sum_{j=1}^n (y_j-x_j) g_{ji}+ \sum_{j:j>i} (x_j-hx_j) f_{ij}-\sum_{k:k<i} (x_k-hx_k) f_{ki} = \nabla_i W(hx,y) - \partial_i W^h\]
and we finish the proof of \eqref{eq:nablahx}.
\end{step}

\begin{step}
We rewrite  \eqref{eq:dup3} using \eqref{eq:nabladiff} and \eqref{eq:nablahx}.

\begin{align*}
& \Big\langle \Dup\Big(\frac{\eta_h^{k+1}}{(k+1)!}\big( a_I \theta_I\theta_{I_h}\big)\Big),
\frac{\partial^{l+2m}(\theta_I\theta_{I_h})}{\partial\theta_{i_1}\cdots\partial\theta_{i_{l+2m}}} 
\partial_{j_0}\partial_{j_1}\cdots \partial_{j_l} \Big\rangle \\
& \nonumber\\
=& \sum_{\stackrel{\sigma\in {\textrm{Perm}} (\{j_0,\cdots,j_l\})}{\tau\in {\textrm{Perm}}(\{i_1,\cdots,i_{l+2m}\})}}
(-1)^{\epsilon_{l+1}} \Big( ( \nabla_{\sigma(j_0)} W - \nabla_{\sigma(j_0)} W(hx,y))
\cdot \Phi_{(\sigma(j_1),\tau( i_1)),\cdots,(\sigma(j_l), \tau(i_l))}^{(\tau(i_{l+1}), \tau(i_{l+2})),\cdots,(\tau(i_{l+2m-1}),\tau (i_{l+2m}))}  \nonumber\\
&\qquad\qquad\qquad\qquad\qquad\qquad
 +\big( \partial_{\tau(i_{l+1})} W^h - \nabla_{\tau(i_{l+1})}W(hx,y)\big)
 \cdot \Phi_{(\sigma(j_0),\tau( i_{l+2})),(\sigma(j_1),\tau(i_1)),\cdots,(\sigma(j_l), \tau(i_l))}^{(\tau(i_{l+3}), \tau(i_{l+4})),\cdots,(\tau(i_{l+2m-1}),\tau (i_{l+2m}))} \Big) \nonumber\\
&\qquad\qquad\qquad\qquad\qquad\qquad
\cdot a_I \frac{\partial^{l+2m} (\theta_I\theta_{I_h})}{\partial\theta_{\tau(i_1)}\cdots\partial\theta_{\tau(i_{l+2m})}} \partial_{\sigma(j_0)}\partial_{\sigma(j_1)}\cdots \partial_{\sigma(j_l)}. \nonumber
\end{align*}
Combining \eqref{eq:ddown1},
\begin{align*}
& \Big\langle \Ddown\Big( \frac{\eta_h^k}{k!}( a_I \theta_I\theta_{I_h})\Big) + \Dup\Big(\frac{\eta_h^{k+1}}{(k+1)!}\big( a_I \theta_I\theta_{I_h}\big)\Big),
\frac{\partial^{l+2m}(\theta_I\theta_{I_h})}{\partial\theta_{i_1}\cdots\partial\theta_{i_{l+2m}}} 
\partial_{j_0}\partial_{j_1}\cdots \partial_{j_l} \Big\rangle \\
& \nonumber\\
=& \sum_{\stackrel{\sigma\in {\textrm{Perm}} (\{j_0,\cdots,j_l\})}{\tau\in {\textrm{Perm}}(\{i_1,\cdots,i_{l+2m}\})}}
(-1)^{\epsilon_{l+1}}  \partial_{\tau(i_{l+1})} W^h 
 \cdot \Phi_{(\sigma(j_0),\tau( i_{l+2})),(\sigma(j_1),\tau(i_1)),\cdots,(\sigma(j_l), \tau(i_l))}^{(\tau(i_{l+3}), \tau(i_{l+4})),\cdots,(\tau(i_{l+2m-1}),\tau (i_{l+2m}))} \\
&\qquad\qquad\qquad\qquad
\cdot a_I \frac{\partial^{l+2m} (\theta_I\theta_{I_h})}{\partial\theta_{\tau(i_1)}\cdots\partial\theta_{\tau(i_{l+2m})}} \partial_{\sigma(j_0)}\partial_{\sigma(j_1)}\cdots \partial_{\sigma(j_l)}. \nonumber
\end{align*}
Rearranging indices, we have
\begin{align*}
& \Ddown\Big( \frac{\eta_h^k}{k!}\big(\sum_I a_I \theta_I\theta_{I_h}\big)\Big) 
+ \Dup\Big(\frac{\eta_h^{k+1}}{(k+1)!}\big(\sum_I a_I \theta_I\theta_{I_h}\big)\Big)\\
=& \sum_{\stackrel{l+m=k}{1\leq i_\bullet,j_\bullet,k_\bullet \leq n}}(-1)^{\epsilon_{l+1}+1} \Phi_{(j_0,i_0),\cdots,(j_l,i_l)}^{(k_1,k_2),\cdots,(k_{2m-3},k_{2m-2})}
\frac{\partial^{l+m-1}}{\partial\theta_{i_0}\cdots\partial\theta_{i_l}\partial\theta_{k_1}\cdots\partial\theta_{k_{2m-2}}}
\Big( \sum_{\stackrel{|I|=r}{i\in I^h}}\partial_i W^h  \cdot a_I \frac{\partial(\theta_I\theta_{I_h})}{\partial\theta_i} \Big)\partial_{j_0}\cdots\partial_{j_l},
\end{align*}
but since $\sum_{|I|=r} a_I \theta_I\theta_{I_h}$ is a cocycle of the Koszul complex of $(\partial_i W^h: i\in I^h)$, we have
\[  \sum_{\stackrel{|I|=r}{i\in I^h}}\partial_i W^h  \cdot a_I \frac{\partial(\theta_I\theta_{I_h})}{\partial\theta_i}= 0.
\]
Therefore,
\[ \Ddown\Big( \frac{\eta_h^k}{k!}\big(\sum_I a_I \theta_I\theta_{I_h}\big)\Big) 
+ \Dup\Big(\frac{\eta_h^{k+1}}{(k+1)!}\big(\sum_I a_I \theta_I\theta_{I_h}\big)\Big)=0\]
and our assertion follows. \qedhere
\end{step}
\end{proof}

%\begin{remark}\label{rmk:rank1rep}
%Combining the above result with Lemma \ref{lemma:homjac}, any cohomology class in $\Hom(\Delta_1,\Delta_h)$ is represented by an element $f(x,y)\cdot \exp(\eta_h)\theta_{I_h}$ for $f\in S$. Furthermore, $f(x,y)\cdot \exp(\eta_h)\theta_{I_h}$ corresponds to $f(x,hx)\cdot \xi_h$ in $\Jac(W^h)\cdot \xi_h$.
%\end{remark}

% We are ready to examine the algebra structure on $\Hom_{1\times G}(\Delta_W^{G\times G},\Delta_W^{G\times G})$. By Lemma \ref{lem:1gequivendo}, we naturally construct the multiplication $\cup$ on $\bigoplus_{g\in G}\Hom(\Delta_1,\Delta_g)$ as follows. Let $\phi \in hom(\Delta_1,\Delta_g)$ and $\psi \in hom(\Delta_1,\Delta_h)$. Then 
% \begin{equation}\label{eq:cupprod}
% \phi \cup \psi:= (h_*\phi) \circ \psi \in hom(\Delta_1,\Delta_{gh}) 
% \end{equation}
% where $h_*\phi: \Delta_h \to \Delta_{gh}$ is defined by
% \[ (h_* \phi)(r(x,y)v_h):= r(x,y)\cdot (1\times h^{-1})\cdot (\phi(\rho(h)v_1)). \]
% The definition of $\cup$ in \eqref{eq:cupprod}  is justified by \eqref{eq:translation}. Namely, given $\phi: \Delta_1 \to \Delta_g$, the $(1\times G)$-equivariant endomorphism of $\Delta_W^{G\times G}$ which restricts to $\phi$ on $\Delta_1$ is $\bigoplus_{h\in G} h_*\phi$. 
% \begin{remark}
% If $\phi=f(x,y)\theta_I \partial_J \in hom(\Delta_1,\Delta_g)$, then 
% \[ h_*\phi= f(x,h^{-1}y)\cdot \rho(h^{-1})(\theta_I\partial_J) \in hom(\Delta_h,\Delta_{hg}).\]
% \end{remark}

It remains to show that $\exp(\eta_h)$ indeed gives a quasi-inverse of $(-)_{kos}$.

\begin{theorem}
Let $\gamma$ be a cocycle of $K^{\bullet} (\partial_{x_i}W^h : i\in I^h) \cdot \theta_{I_h}$. Then we have the following identity of Koszul cohomology classes:
\[ \big[\big(\exp(\eta_h)(\gamma)\big)_{kos}\big]=[\gamma].\]
\end{theorem}

\begin{proof}
Without loss of generality let $\gamma\in K^{-s}(\partial_{x_i}W^h : i\in I^h) \cdot \theta_{I_h}$, so
\[\displaystyle\gamma=\sum_{|I|=s} c_I \theta_I \theta_{I_h}.\]
Decompose $\eta_h=\eta_h^{1,1}+\eta_h^{2,0}$ where
\[ \eta_h^{1,1}(\theta_I \partial_J)=  \sum (-1)^{|I|} g^h_{ji} \frac{\partial\theta_I}{\partial\theta_i}\partial_j\partial_J, \quad \eta_h^{2,0}(\theta_I \partial_J)=\sum f^h_{ji}\frac{\partial^2 \theta_I}{\partial\theta_j \partial\theta_i}\partial_J.\]
Since the operations $\eta_h^{1,1}$ and $\eta_h^{2,0}$ commute, we have
\begin{equation}\label{eq:expdecomp}
\exp(\eta_h)(\gamma)=\sum_n \sum_{k+l=n} \frac{(\eta_h^{1,1})^k}{k!}\cdot \frac{(\eta_h^{2,0})^l}{l!} (\gamma).
\end{equation}
Since $f^h_{ji} \neq 0$ only if $i\in I_h$ or $j\in I_h$, we deduce that
\[ \big(pr\circ \big(\exp(\eta_h)(\gamma)\big)\big)_+=\big(pr\circ \big(\exp(\eta_h^{1,1})(\gamma)\big)\big)_+\]
and the right hand side is explicitly
\begin{equation}\label{eq:prexp}
\big(pr\circ \big(\exp(\eta_h^{1,1})(\gamma)\big)\big)_+ = \gamma+ 
\sum_{k\geq 1}\sum_{\stackrel{i_\bullet,j_\bullet\in I,}{\{i_1,\cdots,i_k\} \cap \{j_1,\cdots,j_k\}=\emptyset}}\pm g^h_{j_1 i_1}\cdots g^h_{j_k i_k}\cdot c_I \frac{\partial^k \theta_I}{\partial\theta_{i_1}\cdots\partial\theta_{i_k}}\theta_{I_h}\cdot 
\partial_{j_1}\cdots\partial_{j_k}. 
\end{equation}
(We did not specify signs in the summation because it will not be needed.)
% Let us denote the summation in \eqref{eq:prexp} by $\gamma_{-}$ so that
% \[ \big(pr\circ \big(\exp(\eta_h^{1,1})(\gamma)\big)\big)_+ = \gamma+ \gamma_{-}.\]

Now we introduce some degree notations. Given a Koszul complex element $\alpha\in K^\bullet(\partial_{x_i} W^h: i\in I^h)\cdot \theta_{I_h}$, let us denote
\[ \alpha= \sum_m \alpha^m \]
where $\alpha^m\in K^{m}(\partial_{x_i}W^h:i\in I^h)\cdot \theta_{I_h}$. 
%Using above quasi-isomorphism $(-)_{kos}$ we also decompose any closed morphism $\phi \in hom(\Delta_1,\Delta_h)$ as 
%\[ \phi=\sum_m \phi^m\]
%such that $(\phi^m)_{kos}=(\phi_{kos})^m$. (So we can denote them just by $\phi_{kos}^m$ without any confusion.) 

When $\phi=\exp(\eta_h)(\gamma)$, we obtain
\[ \phi_{kos}=\sum_{k \geq 0} \phi_{kos}^{-s+2k}\]
where
\begin{align}
\phi_{kos}^{-s+2k}=&
\Big(\sum_{\stackrel{i_\bullet,j_\bullet\in I,}{\{i_1,\cdots,i_k\} \cap \{j_1,\cdots,j_k\}=\emptyset}}
\pm g^h_{j_1 i_1}\cdots g^h_{j_k i_k} \cdot c_I 
\frac{\partial^{2k} \theta_I}{\partial\theta_{i_1}\cdots\partial\theta_{i_k}\partial_{j_1}\cdots\partial_{j_k}}\theta_{I_h}\Big)|_{\Fix(h)}. \label{eq:alphalower}
\end{align}
Clearly $\phi_{kos}^{-s}=\gamma$. Let us show that $\phi_{kos}^{-s+2k}$ is homologous to zero for $k>0$. Consider the following decomposition
\[ \phi=\sum_{k\geq 0} \phi^{-s+2k}\]
where each component $\phi^{-s+2k}$ corresponds to $\phi_{kos}^{-s+2k}$ by the quasi-isomorphism $(-)_{kos}$. Two morphisms $ \phi^{-s+2k}$ and $\big(\exp(\eta_h)(\phi_{kos}^{-s+2k})\big)^{-s+2k}$ of matrix factorizations are homotopic, because their corresponding Koszul cocycles are identical. Consider these morphisms restricted to the submodule $S\subset S\langle \theta_1,\cdots,\theta_n\rangle$ in the following diagram.
\[\displaystyle\xymatrix{
& & & & S \ar[ddllll]_(.3){s_1}\ar@<0.5ex>[dd]^{\phi^{-s+2k}} \ar@<-0.5ex>[dd]_-{(\exp(\eta_h)(\phi_{kos}^{-s+2k}))^{-s+2k}}\ar[rr]^-{\sum_i (y_i-x_i)\theta_i} \ar[ddrr]^(.3){s_2} & & \bigoplus_i S\cdot \theta_i \ar[ddll]^(.2){s_3}\\
& & & & & & \\
\bigoplus\limits_{|J|=s-2k,i\in J\sqcup I_h} S\cdot \theta_{J\sqcup I_h-\{i\}} \ar[rrrr]^-{\sum_i (y_i-hx_i)\theta_i} & & & &
 \bigoplus\limits_{|J|=s-2k}S\cdot \theta_J\theta_{I_h} & & \bigoplus\limits_{|J'|=s-2k+1} S\cdot \theta_{J'}\theta_{I_h} \ar[ll]_-{\sum_{i\in I^h} \nabla_i W(hx,y)\partial_i}}
\]
The restriction of $\phi^{-s+2k}$ to $S$ is $0$ for $k>0$, because any component of $\phi$ which does not consist of contraction operators $\partial_1,\cdots,\partial_n$ is differentiated at least once by $\frac{\partial}{\partial \theta_i}$ for some $i\in I_h$.
On the other hand, $\big(\exp(\eta_h)(\phi_{kos}^{-s+2k})\big)^{-s+2k}$ is restricted to $\phi_{kos}^{-s+2k}$ which is identified as a module map from $S$.
%which can be rewritten from \eqref{eq:alphalower} as
%\[ \phi_{kos}^{-s+2k}=\sum_{|J|=s-2k}f_J \theta_J \theta_{I_h}.\]
By homotopy relation we have
\[ \phi_{kos}^{-s+2k}= \big(\sum_i (y_i-hx_i)\theta_i\big)\cdot s_1 +\big(\sum_{i\in I^h} \nabla_i W(hx,y)\partial_i\big)\cdot s_2 +(-1)^{|s_3|+1} s_3 \cdot \sum_i (y_i-x_i)\theta_i.\]
%so
%\begin{equation}\label{eq:fxy} f(x,y)=\sum_i (y_i-hx_i)s'_{1,i}+\sum_i \nabla_i W(hx,y) s'_{2,i}+\sum_i (y_i-x_i)s'_{3,i} \end{equation}
%for some $s'_{1,i},s'_{2,i},s'_{3,i} \in S$ for $i=1,\cdots,n$.

%Remark \ref{rmk:rank1rep} tells us that $f(x,y)\cdot \exp(\eta_h)\theta_{I_h}$ corresponds to $f(x,hx)\cdot \Omega \in \Jac(W^h)$ via the isomorphism in Lemma \ref{lemma:homjac}. 
Since $\phi_{kos}^{-s+2k}$ depends only on $x_1^h,\cdots,x_n^h$, we have
\[ \phi_{kos}^{-s+2k}=\Big( \big(\sum_i (y_i-hx_i)\theta_i\big)\cdot s_1 +\big(\sum_{i\in I^h} \nabla_{i} W(hx,y)\partial_i\big)\cdot s_2 +(-1)^{|s_3|+1} s_3 \cdot \sum_i (y_i-x_i)\theta_i\Big)|_{\Fix(h)}.\]
It is evident that $(y_i-hx_i)\theta_i |_{\Fix(h)}=(y_i-x_i)\theta_i|_{\Fix(h)}=0$, so 
\[ \phi_{kos}^{-s+2k}=\Big(\big(\sum_{i\in I^h} \nabla_{i} W(hx,y)\partial_i\big)\cdot s_2\Big) |_{\Fix(h)}.\]
By definition of $\nabla_i W$, if $i\in I^h$ then
\[ (\nabla_i W)(x_1^h,\cdots,x_n^h,x_1^h,\cdots,x_n^h)= \partial_{x_i} W^h,\]
so $\phi_{kos}^{-s+2k}$ is in the image of $\sum_{i\in I^h} \partial_{x_i} W^h \cdot \partial_i$, hence is a coboundary element.
\end{proof}

There was a noteworthy fact in the middle of the proof. This will be useful in explicit examples.
\begin{prop}\label{prop:discardpartial}
Let $\phi=\sum_I a_I \theta_I+ \sum\limits_{\stackrel{J,K}{K\neq \emptyset}}b_{JK} \theta_J\partial_K$ be a closed morphism in $hom(\Delta_1,\Delta_h)$. 
Then \[\phi_{kos}= (\sum_I a_I\theta_I)_{kos}.\]
\end{prop}

\begin{definition}
    Let $(W,G)$ be an orbifold LG model (with diagonal $G$-action). Then 
    \[\cup_{kos}: K^\bullet(W,G) \otimes K^\bullet(W,G) \to K^\bullet(W,G),\] 
    defined for $\alpha \in K^\bullet(W,h)$ and $\beta \in K^\bullet(W,h')$ as 
    \begin{equation}\label{def:kosmulti} \alpha \cup_{kos} \beta:= \big(h'_*(\exp(\eta_h)(\alpha)) \cdot \exp(\eta_{h'})(\beta)\big)_{kos},\end{equation}
    induces an associative algebra structure on $H^\bullet(K^\bullet(W,G))$, where for $\phi=f(x,y)\theta_I \partial_J \in hom(\Delta_1,\Delta_{h})$,
\begin{equation}\label{eq:translation}
h'_*\phi:= f(x,h'^{-1}y)\cdot \rho(h'^{-1})(\theta_I\partial_J) \in hom(\Delta_{h'},\Delta_{h'h}).
\end{equation}
\end{definition}
The role of \eqref{eq:translation} is clear: we cannot compose morphisms $\exp(\eta_h)(\alpha)$ and $\exp(\eta_{h'})(\beta)$ as themselves due to the discrepancy of domain and codomain, so we need to translate the latter morphism properly. See \cite{CLe} for more explanation.

We want to emphasize that every step involved in \eqref{def:kosmulti} only requires simple calculus of differential operators and can be made as an automated process (for example in a computer program if needed), though at first glance it looks not so intrinsic due to the detour to the category of matrix factorizations. We need to remark is that the product structure on Koszul cohomology of an orbifold LG model was also defined in \cite{ShkLGorb} via transferring the product structure on a bar complex for Hochschild cohomology of a curved algebra. The advance of our work from previous ones is that our product structure $\cup_{kos}$ is more directly described in nonisolated cases. 

% So far we discovered how to recover an element in $hom(\Delta_1,\Delta_h)$ from an $h$-twisted Koszul cocycle. However we cannot compose $\phi \in hom(\Delta_1,\Delta_{h_1})$ and $\psi \in hom(\Delta_1,\Delta_{h_2})$. The precise way of "multiplying" them is to find $\tilde{\phi}$ and $\tilde{\psi}$ which are endomorphisms of $\Delta_W^{G\times G}$ in the category $MF_{1\times G}(W\boxminus W)$, completely determined by $\phi$ and $\psi$. In \cite{CLe} we have a 

\section{An example in mirror symmetry}\label{sec:example}
Let $\Z/2=\{1,\rho\}$ act on $\C[x_1,x_2,x_3]$ by 
\[ \rho\cdot x_1=-x_1,\; \rho\cdot x_2=-x_2,\; \rho\cdot x_3=x_3\]
and let $W=x_1^2+x_2^2+x_1 x_2 x_3$. We want to remark that $W$ is a mirror of the orbifold sphere $\PP^1_{2,2,\infty}$ (see \cite{CCJ}), which is the $\Z/2$-quotient of $T^* S^1$ (Figure \ref{fig:z2cover}). 
Therefore, $\Z/2$-equivariant theory of $W$ is expected to be deeply related to the symplectic geometry of $T^* S^1$. In this section let us compute the orbifold Koszul algebra structure of $(W,\Z/2)$ and compare it with the symplectic cohomology of $T^* S^1$. Recall that $SH^\bullet(T^* S^1)\cong \C[x,x^{-1},\theta]$ where $\theta$ is a degree 1 element with $\theta^2=0$ (see \cite{Pas_CY} for example).
\begin{figure}
    \centering
    \includegraphics[width=0.9\linewidth]{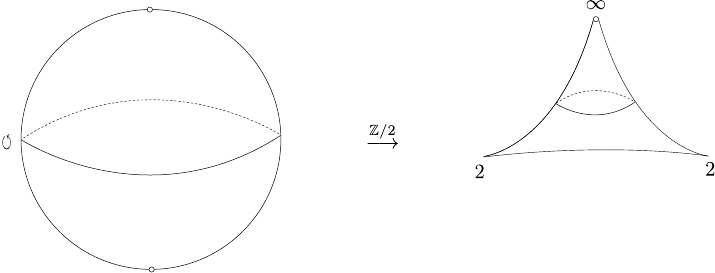}
    \caption{The $\Z/2$-cover of $\PP^1_{2,2,\infty}$ is the 2-punctured sphere, namely $T^* S^1$.}
    \label{fig:z2cover}
\end{figure}

First, the usual Koszul complex $K^\bullet(\partial W)$ is 
\begin{equation}\label{eq:kostrivial}
\begin{split}
\xymatrix{
0 \ar[r] & \C[x_1,x_2,x_3]\theta_1\theta_2\theta_3 \ar[r]^-{d_W} 
& \bigoplus\limits_{1\leq i<j\leq 3}\C[x_1,x_2,x_3]\theta_i\theta_j \ar[r]^-{d_W} 
& \bigoplus\limits_{1\leq i \leq 3}\C[x_1,x_2,x_3]\theta_i } \\
\xymatrix{ \ar[r]^-{d_W} 
& \C[x_1,x_2,x_3] \ar[r] & 0 & & & & & & && }
\end{split}
\end{equation}
where $d_W= (2x_1+x_2 x_3)\partial_1 + (2x_2+x_1 x_3)\partial_2 + x_1 x_2 \partial_3$. Its cohomology is computed as follows (with the aid of Macaulay2 \cite{Macaulay}):
\[ H^i(K^\bullet (\partial W)) \cong 
\begin{cases}
    \frac{\C[x_1,x_2,x_3]}{(2x_1+x_2 x_3,2x_2+x_1 x_3, x_1 x_2)} & {i=0}, \\
    \C[x_3]\cdot \langle 2x_2 \theta_1 - x_2 x_3 \theta_2 +(x_3^2-4)\theta_3\rangle & {i=-1}, \\
    0 & {\textrm{otherwise}},
\end{cases}
\]
and its $\Z/2$-invariant submodule is given by
\[ H^0(K^\bullet (\partial W))^{\Z/2} \cong \C[x_3], \quad H^{-1}(K^\bullet (\partial W))^{\Z/2} = H^{-1}(K^\bullet (\partial W)).\]
To avoid confusion, we will distinguish a cocycle $\alpha$ of \eqref{eq:kostrivial} by $\alpha^1$, meaning that we consider a cocycle in the $1$-sector. Now, since $W^\rho=W(x_1^\rho,x_2^\rho,x_3^\rho)=0$, the $\rho$-twisted Koszul complex $K^\bullet(\partial W^\rho)\cdot \theta_{I_\rho}$ is
\[ 0 \to \C[x_3]\cdot\theta_{1}\theta_{2}\theta_{3} \stackrel{0}{\longrightarrow} \C[x_3]\cdot\theta_{1}\theta_{2} \to 0.\]
Again, we will denote a cocycle $\beta$ of this $\rho$-sector complex by $\beta^\rho$. Observe that $H^*(K^*(W,\Z/2))^{\Z/2}$ is a free $\C[x_3]$-module.

Let us compute $\exp(\eta_\rho)((\theta_1\theta_2)^\rho)$ as follows. First,
\begin{align*}
    \barW^\rho_{0,0}&=W(\rho x_1, \rho x_2, \rho x_3)= x_1^2+x_2^2+x_1 x_2 x_3, &
    &\barW^\rho_{0,1}=W(x_1,\rho x_2,\rho x_3)= x_1^2+x_2^2-x_1 x_2 x_3, \\
    \barW^\rho_{0,2}&=W(x_1,x_2,\rho x_3)=x_1^2+x_2^2+x_1 x_2 x_3,&
    & \barW^\rho_{0,3}=W(x_1,x_2,x_3)=x_1^2+x_2^2+x_1 x_2 x_3, \\
    \barW^\rho_{1,1}&=W(y_1, \rho x_2, \rho x_3)= y_1^2+x_2^2-x_1 x_2 x_3, &
    &\barW^\rho_{1,2}=W(y_1,x_2,\rho x_3)= y_1^2+x_2^2+x_1 x_2 x_3, \\
    \barW^\rho_{1,3}&=W(y_1,x_2, x_3)=y_1^2+x_2^2+x_1 x_2 x_3,&
    & \barW^\rho_{2,2}=W(y_1,y_2,\rho x_3)=y_1^2+y_2^2+y_1 y_2 x_3, \\
    \barW^\rho_{2,3}&=W(y_1, y_2, x_3)= y_1^2+y_2^2+y_1 y_2 x_3, &
    &\barW^\rho_{3,3}=W(y_1,y_2,y_3)= y_1^2+y_2^2+y_1 y_2 y_3,
\end{align*}
and likewise, 
\begin{align*}
    \tildeW^\rho_{0,0}&=W(\rho x_1, \rho x_2, \rho x_3)= x_1^2+x_2^2+x_1 x_2 x_3, &
    &\tildeW^\rho_{0,1}=W(x_1,\rho x_2,\rho x_3)= x_1^2+x_2^2-x_1 x_2 x_3, \\    \tildeW^\rho_{0,2}&=\tildeW^\rho_{0,3}=x_1^2+x_2^2+x_1 x_2 x_3, & & \tildeW^\rho_{1,1}=\tildeW^\rho_{1,2}=\tildeW^\rho_{1,3}=x_2^2,  \\
     \tildeW^\rho_{2,2}&=\tildeW^\rho_{2,3}=\tildeW^\rho_{3,3}=0.& &
\end{align*}
% Actually, we do not need $\barW^\rho_{j,3}$ and $\tildeW^\rho_{j,3}$ if we only compute $\exp(\eta_\rho)((\theta_1\theta_2)^\rho)$ (they are needed for $\exp(\eta_\rho)((\theta_1\theta_2\theta_3)^\rho)$). 
Plugging $\barW^\rho_{j,i}$ and $\tildeW^\rho_{j,i}$ into \eqref{eq:gjifji} and \eqref{eq:gii}, we have
\begin{align*}
    & g^\rho_{11}=1,\quad g^\rho_{12}=x_3,\quad g^\rho_{13}=x_2,\quad g^\rho_{22}=1,\quad g^\rho_{23}=y_1,\quad g^\rho_{33}=0, \\
    & f^\rho_{12}=-\frac{x_3}{2}, \quad f^\rho_{13}=-\frac{x_2}{2},\quad f^\rho_{23}=0.
\end{align*}
By \eqref{etasign} and \eqref{eq:exp} it directly follows that
\[ \exp(\eta_\rho)((\theta_1\theta_2)^\rho)=\theta_1\theta_2+\theta_2\partial_1-x_3\theta_1\partial_1-\theta_1\partial_2+\partial_1\partial_2+\frac{x_3}{2}.\]

Consequently
\begin{align*}
 &   \rho_* \big(\exp(\eta_\rho)((\theta_1\theta_2)^\rho)\big) \cdot \exp(\eta_\rho)((\theta_1\theta_2)^\rho) \\
= & (\theta_1\theta_2+\theta_2\partial_1-x_3\theta_1\partial_1-\theta_1\partial_2+\partial_1\partial_2+\frac{x_3}{2}) \cdot (\theta_1\theta_2+\theta_2\partial_1-x_3\theta_1\partial_1-\theta_1\partial_2+\partial_1\partial_2+\frac{x_3}{2}) \\
=& \frac{x_3^2}{4}-1+ c_1\cdot\partial_1+c_2\cdot\partial_2+c_{12}\cdot\partial_1\partial_2
\end{align*} 
where $c_1,c_2,c_{12}\in S\langle \theta_1,\theta_2,\theta_3\rangle$. By Proposition \ref{prop:discardpartial}, we have 
\begin{equation}\label{eq:squareoftheta12} 
\big(\rho_* \big(\exp(\eta_\rho)((\theta_1\theta_2)^\rho)\big) \cdot \exp(\eta_\rho)((\theta_1\theta_2)^\rho)\big)_{kos} = \big(\frac{x_3^2}{4}-1\big)^1.
\end{equation}
Likewise, we can deduce that
\begin{equation}\label{eq:12times123} 
\big(\rho_*\big(\exp(\eta_\rho)((\theta_1\theta_2)^\rho)\big) \cdot \exp(\eta_\rho)((\theta_1\theta_2\theta_3)^\rho)\big)_{kos}= \Big(\frac{x_2}{2}\theta_1-\frac{x_2 x_3}{4}\theta_2+\big(\frac{x_3^2}{4}-1\big)\theta_3\Big)^1.
\end{equation}
Above multiplications \eqref{eq:squareoftheta12} and \eqref{eq:12times123} on $\rho$-sector are the most nontrivial ones. Besides them, commutativity of $\cup_{kos}$ on $K^\bullet(W,\Z/2)^{\Z/2}$ are readily checked by the recipe (the result on commutativity on general orbifold Koszul algebras will appear elsewhere). The following is now straightforward.
\begin{theorem}
    $H^\bullet(K^\bullet(W,\Z/2))^{\Z/2} \cong SH^\bullet(T^* S^1)$.
\end{theorem}

\begin{proof}
Recall that $SH^\bullet(T^* S^1) \cong \C[x,x^{-1}][\theta]$ where $\theta^2=0$. Observe that $\C[x,x^{-1}][\theta]$ is isomorphic to a quotient of free commutative algebra:
\[ \C[x,x^{-1}][\theta] \cong \C[\alpha,\beta,\gamma]/(\alpha\beta-1,\gamma^2).\]
Define an algebra homomorphism from the free commutative algebra as
\begin{align*} 
 \mu: \C[\alpha,\beta,\gamma] & \to H^\bullet(K^\bullet(W,\Z/2))^{\Z/2},\\
\alpha &\mapsto \lambda_+:= \big(\frac{x_3}{2}\big)^1+(\theta_1\theta_2)^\rho,\\
 \beta & \mapsto \lambda_-:= \big(\frac{x_3}{2}\big)^1-(\theta_1\theta_2)^\rho, \\
 \gamma & \mapsto \tau:= (\theta_1\theta_2\theta_3)^\rho.
 \end{align*}
$\mu$ is surjective due to
\[ \mu(\alpha+\beta)=(x_3)^1,\quad \mu\big(\frac{\alpha-\beta}{2}\big)= (\theta_1\theta_2)^\rho,
\quad \mu(\gamma)=(\theta_1\theta_2\theta_3)^\rho,\]
\[\quad \mu(\frac{\alpha-\beta}{2}\cdot \gamma)= (\theta_1\theta_2)^\rho \cup_{kos}(\theta_1\theta_2\theta_3)^\rho = \Big(\frac{x_2}{2}\theta_1-\frac{x_2 x_3}{4}\theta_2+\big(\frac{x_3^2}{4}-1\big)\theta_3\Big)^1,\]
meaning that we recover all generators of $H^\bullet(K^\bullet(W,\Z/2))^{\Z/2}$ by $\mu$. We also check that \[\lambda_+ \cup_{kos} \lambda_-=1, \quad \tau\cup_{kos}\tau=0\] completely determine the relation of multiplications between them, hence \[\ker\mu = (\alpha\beta-1,\gamma^2).\qedhere\]
% such that for any $n,m,k \geq 0$,
% \begin{equation}\label{eq:ringhom}
% \mu(\alpha^n \beta^m \gamma^k)=(\lambda_+)^n(\lambda_-)^m \gamma^k
% \end{equation}
% (on the right hand side of \eqref{eq:ringhom} the product is with respect to $\cup_{kos}$). 
\end{proof}
%To compare $K^\bullet(W,\Z/2)^{\Z/2}$ with $\C[x,x^{-1},\theta]$ we need a nonconstant element which is a unit. 
% Let $\lambda_+,\lambda_- \in K^\bullet(W,\Z/2)^{\Z/2}$ be defined by
% \[\lambda_+:=\big(\frac{x_3}{2}\big)^1+(\theta_1\theta_2)^\rho,\quad \lambda_-:=\big(\frac{x_3}{2}\big)^1-(\theta_1\theta_2)^\rho.\]
% By commutativity we have $\lambda_+ \cup_{kos} \lambda_-=1.$ This explains a similarity of $K^\bullet(W,\Z/2)^{\Z/2}$ with $\C[x,x^{-1},\theta]$. Rigorous construction of an isomorphism is left to the reader.

% It follows that
% \begin{align*}
% & x_3 \cup_{kos} \theta_1\theta_2 = \theta_1\theta_2 \cup_{kos} x_3 = x_3 \theta_1\theta_2,\\
% & x_3 \cup_{kos} \theta_1\theta_2\theta_3 = \theta_1\theta_2 \theta_3\cup_{kos} x_3 = x_3 \theta_1\theta_2\theta_3,
% \end{align*}

\bibliographystyle{amsalpha}
\bibliography{geometry}

\end{document}